\documentclass[12pt]{article}

\usepackage[margin=2cm]{geometry}
\usepackage[utf8]{inputenc}
\usepackage{amsmath}
\usepackage{amssymb}
\usepackage{amsthm}
\usepackage{color} 
\usepackage{dsfont}
\usepackage{wasysym}
\usepackage{graphicx}
\usepackage{hyperref}

\newtheorem{definition}{Definition}
 
\newtheorem{lemma}[definition]{Lemma} 
\newtheorem{theorem}[definition]{Theorem}

\newtheorem{observation}[definition]{Observation}

\newcommand{\R}{\mathds{R}}

\newcommand{\nula}{\mathbf{0}}

\newcommand{\LCP}{\mathrm{LCP}}

\newcommand{\qq}{\mathbf{q}}
\newcommand{\xx}{\mathbf{x}}
\newcommand{\yy}{\mathbf{y}}
\newcommand{\zz}{\mathbf{z}}

\newcommand{\bb}{\mathbf{b}}

\newcommand{\ww}{\mathbf{w}}

\newcommand{\F}{\mathcal{F}}
\renewcommand{\O}{\mathcal{O}}

\renewcommand{\L}{\mathcal{L}}
\newcommand{\I}{\mathcal{I}}

\renewcommand{\vert}{\mathop{\rm vert}}
\newcommand{\carr}{\mathop{\rm carr}}

\title{A New Combinatorial Property of Geometric Unique Sink Orientations}
\author{Yuan Gao\thanks{Institute of Theoretical Computer Science, Department of Computer Science, ETH Zurich, Zurich, Switzerland} \and Bernd G\"artner\thanks{Institute of Theoretical Computer Science, Department of Computer Science, ETH Zurich, Zurich, Switzerland} \and Jourdain Lamperski\thanks{Operations Research Center, Massachusetts Institute of Technology, Cambridge MA, USA} }
\date{\today}

\begin{document}
\maketitle

\begin{abstract}
  A unique sink orientation (USO) is an orientation of the hypercube graph with the property that every face has a unique sink. A number of well-studied problems reduce in strongly polynomial time to finding the global sink of a USO; most notably, linear programming (LP) and the P-matrix linear complementarity problem (P-LCP). The former is not known to have a strongly polynomial-time algorithm, while the latter is not known to even have a polynomial-time algorithm, motivating the problem to find the global sink of a USO. Although, every known class of \emph{geometric} USOs, arising from a concrete problem such as LP, is exponentially small, relative to the class of all USOs. Accordingly, geometric USOs exhibit additional properties that set them apart from general USOs, and it may be advantageous, if not necessary, to leverage these properties to find the global sink of a USO faster. Only a few such properties are known. In this paper, we establish a new combinatorial property of the USOs that arise from \emph{symmetric} P-LCP, which includes the USOs that arise from linear and simple convex quadratic programming.
\end{abstract}

\section{Introduction}
A \emph{unique sink orientation} (USO) is an orientation of the $n$-dimensional
hypercube graph (the $n$-cube) with the property that every face (subcube) has a
unique sink. Figure~\ref{fig:uso} depicts a 3-dimensional USO 
called the \emph{spinner} and highlights
the unique sink of the whole cube (which is itself a face) and the
unique sink of each of the six 2-faces. 

\begin{figure}[htb]
\begin{center}  
  \includegraphics[width=0.2\textwidth]{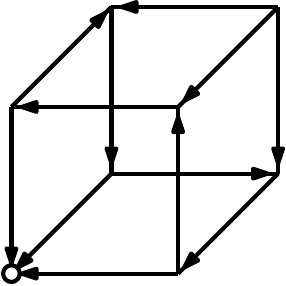}\hspace{1cm}
  \includegraphics[width=0.3\textwidth]{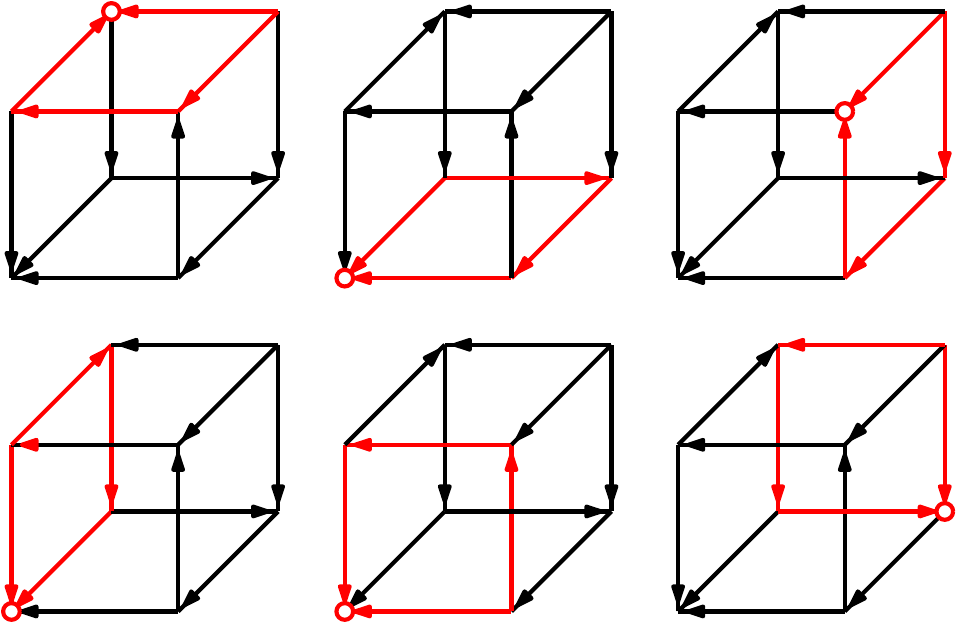}\hspace{1cm}
  \includegraphics[width=0.2\textwidth]{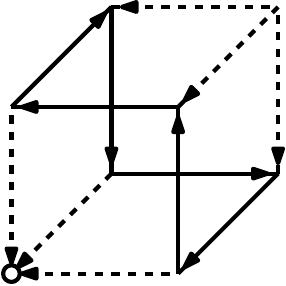}
\end{center}
\caption{The spinner, a 3-dimensional USO; the black circle is the unique global sink, and the red circles are the unique sinks in the six 2-faces. 
The USO is called the spinner because it contains a directed cycle of length~6 and is rotationally symmetric under any cyclic permutation of the dimensions.\label{fig:uso}} 
\end{figure}

USOs were introduced by Stickney and Watson as a combinatorial
model to study principal pivoting algorithms for P-matrix linear
complementarity problems (P-LCPs)~\cite{stickney}. The USOs resulting
from P-LCPs are called \emph{P-cubes}. The spinner in Figure~\ref{fig:uso} is actually 
Stickney and Watson's \emph{digraph 19}.

After having been forgotten
for more than twenty years, USOs were rediscovered by Sz\'abo and
Welzl~\cite{szabo2001unique}, motivated by applications in
optimization theory, but mostly by their simple and clean combinatorial
structure.

\paragraph{The algorithmic problem.}
The connection to optimization is as follows. The unique
global sink $S$ corresponds to an optimal solution, among a set of candidate
solutions (the $2^n$ cube vertices), to some optimization problem. The
USO defines a structure on the search space that is not explicitly
given (because of its exponential size) but can be queried
locally. That is, we have a \emph{vertex evaluation} oracle that takes as input any vertex
$V$ and returns the orientations of the $n$ edges incident to $V$. The
algorithmic question is: how many vertex evaluations do we need to find
$S$? In applications, vertex evaluation can typically be performed in polynomial time.

Clearly we can find the global sink with $2^n$ vertex evaluations; however, we can do better. Prior to Sz\'abo and Welzl, a number of acyclic combinatorial models had been studied. These models arise when candidate solutions are ordered by an objective function. In such models, \emph{subexponential} randomized algorithms exist~\cite{MatShaWel:A-subexponential,g-saaop-95}; the main ideas involved can essentially be distilled into the \emph{Random Facet} algorithm, which requires an expected number of at most $e^{2\sqrt{n}}$ vertex evaluations to find the global sink of an acyclic USO~\cite{Gar:The-random-facet}. However, USOs may in general contain directed cycles (see Figure~\ref{fig:uso} for the smallest example); in particular, P-cubes may contain directed cycles. Sz\'abo and Welzl were the first to come up with nontrivial algorithms for this more general context. They developed a deterministic algorithm that uses at most $1.606^n$ vertex evaluations~\cite[Theorem 4.1]{szabo2001unique} and a randomized algorithm~\cite[Lemma 3.2]{szabo2001unique} that uses at most $1.438^n$ vertex evaluations in expectation when combined with the optimal randomized algorithm for the $3$-cube~\cite{Tessaro}.

The aforementioned algorithmic results have not be improved during the last twenty years, although there seems to be a lot of room for improvement. Indeed, the best known lower bound for the number of vertex evaluations is a small polynomial $\Omega(n^2/\log n)$~\cite{SchSza:Finding}.

\paragraph{Geometric USOs.} 
It might be that USOs are algorithmically difficult because they form an extremely rich class of orientations that contains many complicated orientations that we never encounter ``in practice''. Let us elaborate on this. As the $n$-cube has $n2^{n-1}$ edges, the number of $n$-cube orientations is $2^{n2^{n-1}}$. The number of USOs is somewhat smaller but still doubly exponential, namely $2^{\Theta(2^n\log n)}$~\cite{Mat:The-Number}. In contrast, the number of USOs that we encounter in applications is tiny, as we explain next. Application areas include linear complementarity (the original source of USOs~\cite{stickney}), linear programming~\cite{gs-lpuso-06,Gar:The-random-facet}, and computational geometry~\cite{fg-sebbcsa-04}. In all of these areas, a USO is typically defined by polynomially (in $n$) many input numbers; for example, a matrix $M\in\R^{n\times n}$ and a vector $\qq\in\R^n$, in the case of P-cubes. More precisely, the edge orientations are determined by signs of algebraic expressions (all of polynomial degree) over the input numbers. In this situation, the number of USOs that can be obtained from all possible inputs is singly exponential. For example, there are only $2^{\Theta(n^3)}$ P-cubes~\cite{FONIOK2014155}. Let us say that a \emph{geometric} USO is one that is generated from an input of polynomial size, as previously described.

Hence, the number of geometric USOs (the ones that we actually want to work with) is exponentially smaller than the number of all USOs. It is therefore natural to ask whether we can exploit this fact algorithmically. That is, why bother with all USOs if we are only interested in a tiny fraction of them? It is also natural to ask if the USO abstraction is even worth considering given that it significantly overestimates the number of problem instances. There is an important reason for the abstraction: a simple combinatorial model might allow us to see and exploit relevant structure of a concrete problem that otherwise would remain hidden behind the input numbers. A striking example of this is the  subexponential algorithm (mentioned before) in the combinatorial model of \emph{LP-type problems}~\cite{ShaWel:A-combinatorial,MatShaWel:A-subexponential}. Although developed in a completely abstract setting, it was at the same time a breakthrough result in the theory of linear programming. 

On the other hand, a combinatorial model typically fails to capture some properties of the underlying geometric situation and therefore loses information. A well-known example is Pappus's theorem that no longer holds when we move from a geometric setting (vector configurations) to a combinatorial setting (matroids)~\cite[Proposition 6.1.10]{OxleyJ.G2011MT}. Still, we might be able to regain some information by extracting additional \emph{combinatorial} structure from the geometric situation. In the remainder of this introduction, we review the known results in this direction for geometric USOs and introduce the new combinatorial property that we establish in this paper.

\paragraph{P-cubes satisfy the Holt-Klee property.}
A classical example of a combinatorial model is a \emph{polytope graph}, a graph formed by the vertices and edges of a polytope. It turns out that many questions about polytopes can be answered through a sufficient understanding of polytope graphs and their properties. An early related result is Balinski's theorem: the graph of an $n$-dimensional polytope is $n$-connected~\cite{balinski1961}. Holt and Klee later established a directed version of Balinski's theorem: in the graph of an $n$-dimensional polytope, with edges oriented according to a generic linear function, there are $n$ directed paths from the highest to the lowest vertex that are internally disjoint~\cite{HolKle:A-proof}. The \emph{Holt-Klee property} generalizes from geometric cubes to P-cubes. For P-cubes, it states that there are $n$ directed paths from the (unique) source to the unique sink that are internally disjoint~\cite[Corollary 4.4]{GarMorRus:Unique}. Figure~\ref{fig:HK} shows such paths for the spinner (which is actually a P-cube; see Section~\ref{sec:P}), along with a USO that fails to satisfy the Holt-Klee property.

\begin{figure}[htb]
\begin{center}  
  \includegraphics[width=0.5\textwidth]{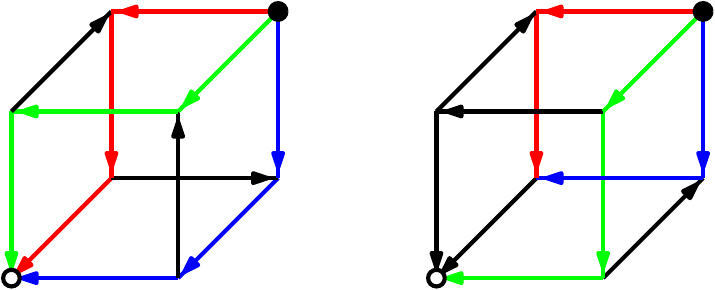}
\end{center}
\caption{The spinner on the left satisfies the Holt-Klee property because there are $3$ internally disjoint directed paths from the global source to the global sink. The USO on the right does not satisfy the Holt-Klee property because the red and the blue path necessarily interfere.\label{fig:HK}} 
\end{figure}

A USO that satisfies the Holt-Klee property is not necessarily a P-cube; the Holt-Klee property is necessary but not sufficient. In fact, it dramatically fails to be sufficient because the number of USOs that satisfy the Holty-Klee property is again doubly exponential~\cite{FONIOK2014155}. Although, it might be the case that such USOs, despite the large number of them, are algorithmically easier than general USOs. But so far, there is no real indication of this. The only known result is as follows. Matou\v{s}ek showed that the Random Facet algorithm may actually need an expected number of $e^{\Theta(\sqrt{n})}$ vertex evaluations for a class of (acyclic) USOs~\cite{Mat:Lower, Gar:The-random-facet}. The lower bound is attained for a random member of the class. In contrast, all members of the class with the Holt-Klee property are solved by Random Facet in expected $O(n^2)$ vertex evaluations~\cite{Gar:The-random-facet}.

\paragraph{K-cubes are locally uniform.}
\emph{K-cubes} are the USOs that arise from K-matrix linear complementarity problems (K-LCPs). K-LCPs form an ``easy'' subclass of P-LCPs that have long known to be solvable in polynomial time (for P-LCPs, this is open)~\cite{Cha:A-special,Sai:A-note}. K-cubes were shown to be \emph{locally uniform}~\cite{FonFukGar:Pivoting}. Identifying the cube vertices with the subsets of $[n]$, this means the following: if some vertex $V$ has only incoming 
edges from (or only outgoing edges to) $k$ ``higher'' neighbors $V\cup\{i_1\},\ldots, V\cup\{i_k\}$, then all edges in the $k$-face spanned by these vertices have the same direction (all down, or all up); see Figure~\ref{fig:locally_uniform}.

\begin{figure}[htb]
\begin{center}  
  \includegraphics[width=0.6\textwidth]{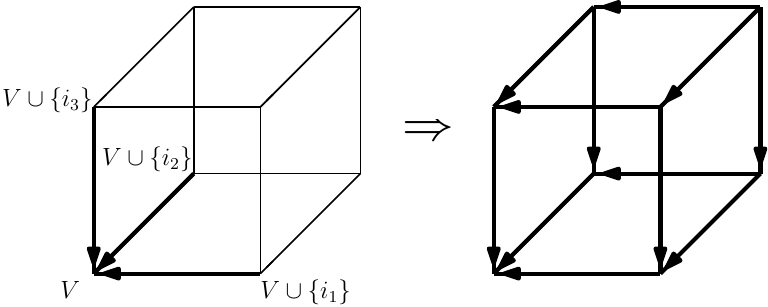}
\end{center}
\caption{Locally uniform USOs: whenever the lowest vertex $V$ in a face has only incoming (or only outgoing) edges to its neighbors in the face, then all edges in the face point downwards (or upwards). \label{fig:locally_uniform}} 
\end{figure}

This property can in fact be exploited algorithmically: in a locally uniform USO, every directed path has length at most $2n$~\cite[Theorem 5.9]{FonFukGar:Pivoting}. Hence, we can find the sink with a linear number of vertex evaluations, starting at any vertex, and simply following a directed path. When applied to K-cubes, this simply enhances the arsenal of available polynomial-time methods for K-LCP, but the number of locally uniform USOs is doubly exponential and therefore much larger than the number of K-cubes~\cite{FONIOK2014155}.

\paragraph{Tridiagonal and Hessenberg P-LCPs.}
A matrix $M=(m_{ij})\in\R^{n\times n}$ is \emph{tridiagonal} if $m_{ij}=0$ whenever $|i-j|>1$, and more generally (lower) \emph{Hessenberg} if $m_{ij}=0$ whenever $j-i>1$. P-LCPs with Hessenberg matrices can be solved in polynomial time~\cite{GARTNER2012484}. Although this can be shown without considering USOs; the main idea involved is a combinatorial property of the corresponding USOs~\cite{Sprecher}.

\paragraph{Our contribution: D-cubes satisfy property~L.}
There is another and actually quite interesting subclass of P-LCPs, namely the ones with a symmetric P-matrix. As the symmetric P-matrices are exactly the positive \textbf{d}efinite matrices, we refer to the USOs generated by symmetric P-LCPs as \emph{D-cubes}. Unlike general P-LCPs, symmetric P-LCPs are equivalent to a class of (strongly) convex quadratic programs~\cite[Section 1.2]{CotPanSto:LCP} and as such can be solved in polynomial time using the ellipsoid method~\cite{KOZLOV1980223}. Still, unlike K-LCPs, the symmtric P-LCPs cannot be called an ``easy'' subclass of all P-LCPs. In particular, no \emph{strongly} polynomial-time algorithms are known; these are algorithms with a number of arithmetic operations that can be polynomially bounded in $n$. This means, the situation is the same as for linear programming, where the problem of finding a strongly polynomial-time algorithm is on Steve Smale's 1998 list of \emph{mathematical problems for the next century}~\cite{Smale}.

Promising candidates for strongly polynomial-time algorithms are \emph{combinatorial methods} that---unlike the ellipsoid method---explore a discrete search space, where every search step can be implemented in strongly polynomial time. The most prominent combinatorial method for linear programming is the simplex method. But as there is a (strongly polynomial-time) reduction from linear programming to the problem of finding the sink in a D-cube~\cite{gs-lpuso-06}, methods for the latter problem are also relevant for the former. In fact, the fastest deterministic combinatorial algorithm (currently known) for solving linear programs with $2n$ constraints and $n$ variables is based on this reduction~\cite{gs-lpuso-06}. 

The new combinatorial property of D-cubes that we establish is as follows. For every vertex $V$ in a USO, and for all $i,j\notin V$, consider the $2$-face spanned by $V,V\cup\{i\},V\cup\{j\},V\cup\{i,j\}$; see Figure~\ref{fig:propertyL}. We write $i\rightarrow j$ if in this face, the edges along dimension $j$ have opposite directions. As the face has a unique sink, it is easy to see that we cannot simultaneously have $j\rightarrow i$. But we may have neither of the two directions.

\begin{figure}[htb]
\begin{center}  
  \includegraphics[width=0.6\textwidth]{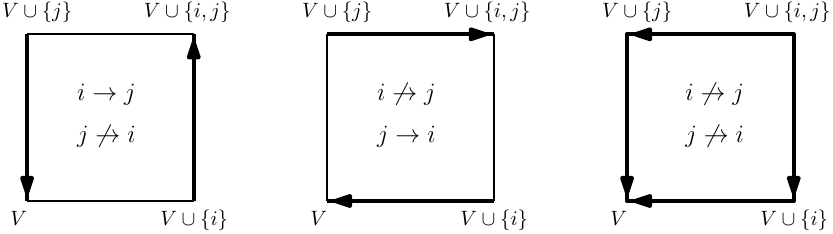}
\end{center}
\caption{Property L at vertex $V$ of a USO. \label{fig:propertyL}}
\end{figure}

This defines a directed graph on the complement of $V$, the \emph{L-graph of $V$}, and we say that a USO has \emph{property~L} if for each of its vertices $V$, the L-graph of $V$ is acyclic. Our main result is that every D-cube has property~L. For example, this implies that the spinner in Figure~\ref{fig:uso}, despite being a P-cube, cannot be a D-cube, as property~L is violated at $\emptyset$ (the lower-left vertex). We also show that the number of USOs having property~L is doubly exponential, so once more, the combinatorial property is far from characterizing the geometric situation.

The paper is organized as follows. In Section~\ref{sec:L}, we formally introduce cube orientations and the L-graphs of such orientations. Section~\ref{sec:uso} reviews USOs and shows that if a cube orientation satisfies property~L, it is already a USO. This may come as a surprise since property~L only involves small local graphs, whereas the USO property involves (exponentially large) faces. We also establish a doubly exponential lower bound on the number of USOs with property~L. In Section~\ref{sec:P}, we introduce P-cubes and in particular review their construction from P-LCPs, due to Stickney and Watson~\cite{stickney}. Section~\ref{sec:D} contains our main result, a proof that every D-cube has property~L. In Section~\ref{sec:Dmisc}, 
we show that, starting from dimension $4$, there are USOs with no isomorphic copy satisfying property~L, and we provide systematic constructions of such USOs. This shows that property~L is \emph{combinatorially nontrivial}: it cannot always be attained by simply relabeling cube vertices.
But in the $3$-dimensional case, an isomorphic copy with property~L can always be found.
Section~\ref{sec:conclusion} provides open problems and further research questions.

\section{Cube orientations, L-graphs and property~L}\label{sec:L}
We closely follow the notation of Sz\'abo and Welzl~\cite{szabo2001unique}. For sets $U,V$, define the \emph{symmetric difference} $U\oplus V = (U\cup V)\setminus (U\cap V)$. For $U\subseteq W$, define the \emph{interval} $[U,W] = \{V: U\subseteq V\subseteq W\}$. 

A \emph{cube orientation} is a directed graph $\O$ with vertex set $\vert\O=[U,W]$, for some interval $[U,W]$, that contains exactly one of the directed edges $(V,V\oplus\{i\})$ and
$(V\oplus\{i\},V)$ for every $V\in\vert\O$ and every $i\in \carr\O := W\setminus U$ (the \emph{carrier} of $\O$). We also write $V\rightarrow_{\O}V'$ if $\O$ contains the directed edge $(V,V')$, or simply $V\rightarrow V'$ if $\O$ is clear from the context. The \emph{dimension} of $\O$ is $|\carr\O|$. We call $\O$ an \emph{$n$-cube orientation} if $\vert\O=[\emptyset,[n]]$, where $[n]:=\{1,2,\ldots,n\}$.

The \emph{outmap} of $\O$ is the function $\phi_{\O}:\vert\O\rightarrow 2^{\carr\O}$ defined by
\[
  \phi_{\O}(V) = \{i\in\carr\O: V \rightarrow_{\O} V\oplus\{i\}\}, \quad V\in\vert\O.
\]
Hence, the set $\phi_{\O}(V)$ contains the cube dimensions along which $V$ has outgoing edges in $\O$.

Next we define the central objects of this paper.

\begin{definition}
  Let $\O$ be a cube orientation. For each vertex $V\in\vert\O=[U,W]$, the \emph{L-graph of $V$}, 
  denoted by $\L_{\O}(V)$, is the directed graph with vertex set $W\setminus V$, and with an arc
  $(i,j)$ for $i, j\in W\setminus V, i\neq j$, whenever
  \begin{equation}\label{eq:Lgraph}
    j\in \phi_{\O}(V)\oplus \phi_{\O}(V\cup\{i\}).
\end{equation}
\end{definition}

In words, the L-graph of $V$ contains the arc $(i,j)$ if exactly one of $V$ and $V\cup\{i\}$ has an outgoing edge along dimension~$j$; see Figure~\ref{fig:Lgraph}. Note that the term ``directed edge'' is used in connection with cube orientations, while ``arc'' is used in connection with L-graphs.

\begin{figure}[htb]
\begin{center}  
  \includegraphics[width=0.5\textwidth]{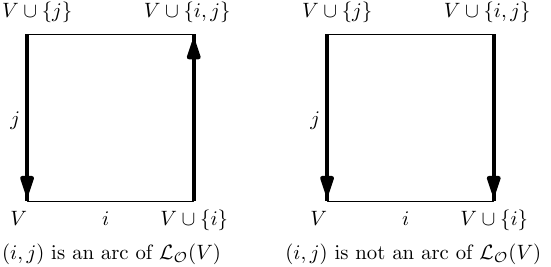}
\end{center}
\caption{Definition of the L-graph.\label{fig:Lgraph}} 
\end{figure}

The graph $\L_{\O}(W)$ is empty, and all graphs of the form $\L_{\O}(W\setminus\{i\})$ are a single vertex. Hence, L-graphs are interesting only when $|V|\leq |W|-2$. Figure~\ref{fig:Lgraphs} shows a 4-cube orientation and four of its L-graphs. 
In later figures, we draw the arcs of the L-graphs as red (double) arrows between the corresponding dimensions of the cube orientation instead of explicitly drawing the L-graphs; see Figure \ref{fig:2dcubes}. 

\begin{figure}[htb]
\begin{center}  
  \includegraphics[width=0.8\textwidth]{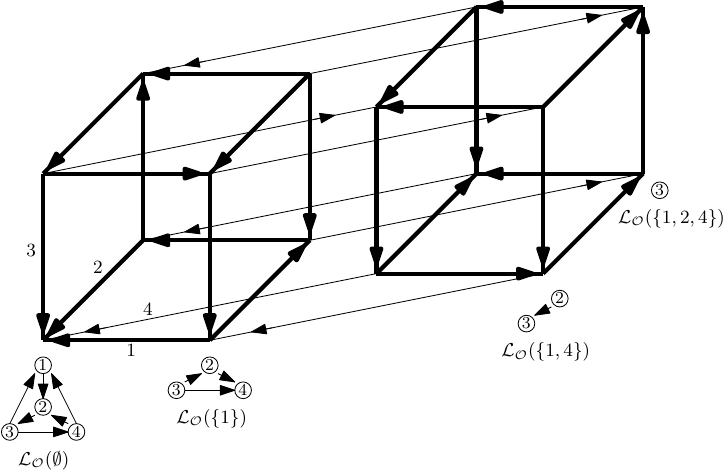}
\end{center}
\caption{A 4-cube orientation and four of its 16 L-graphs. \label{fig:Lgraphs}} 
\end{figure}

The arc(s) connecting $i$ and $j$ in $\L_{\O}(V)$ are determined by a $2$-dimensional cube orientation. Up to isomorphism, there are four different such orientations: the \emph{eye}, the \emph{bow}, the \emph{twin peak}, and the \emph{cycle}; see Figure~\ref{fig:2dcubes}. The terms eye and bow are due to Szab\'o and Welzl~\cite[Figure 4]{szabo2001unique}.

\begin{figure}[htb]
\begin{center}  
  \includegraphics[width=\textwidth]{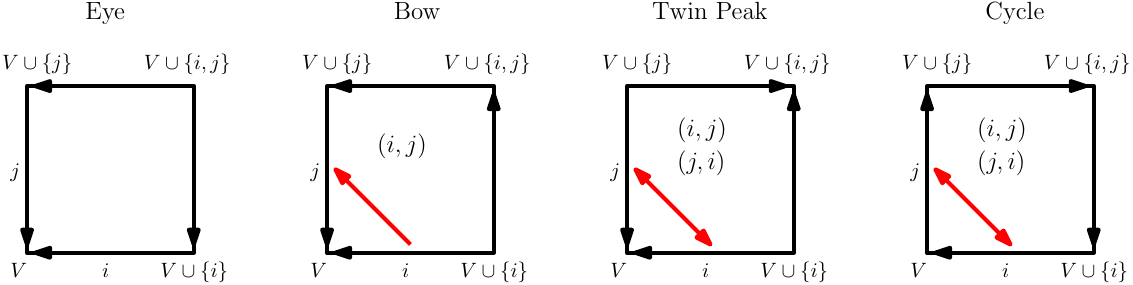}
\end{center}
\caption{Possible arcs between $i$ and $j$ in $\L_{\O}(V)$. \label{fig:2dcubes}} 
\end{figure}

For an eye, there are no arcs between $i$ and $j$. For a bow, there is exactly one of the arc, and for the twin peak and the cycle, both arcs are present. 

Figure~\ref{fig:3dcubes} depicts two $3$-cube orientations and their (nontrivial) L-graphs of
$V=\emptyset,\{1\},\{2\},\{3\}$. We see here that the left one, the spinner, has a cyclic L-graph, while the right one does not. This may be surprising, given that the two orientations are isomorphic. L-graphs have a ``sense of direction'': for every vertex $V$, we only use the ``higher-dimensions'' $i\notin V$ to define $\L_{\O}(V)$. Under an automorphism, a higher-dimension at $V$ may become a lower-dimension in the isomorphic image of $V$, and vice versa. In fact, the automorphism in Figure~\ref{fig:3dcubes} precisely flips the higher-lower status of dimension $1$ at all vertices. 

\begin{figure}[htb]
\begin{center}  
  \includegraphics[width=0.6\textwidth]{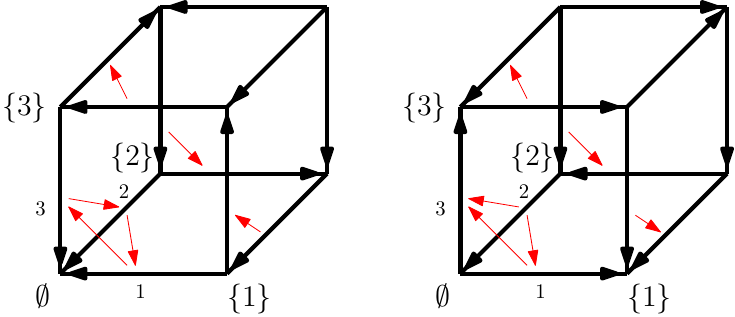}
\end{center}
\caption{L-graphs of two isomorphic 3-cube orientations may not be isomorphic. \label{fig:3dcubes}} 
\end{figure}

Here is the central definition of this paper:

\begin{definition}
An $n$-cube orientation $\O$ has \emph{property~L} if all of its L-graphs are acyclic.
\end{definition}

Consider the 2-cube orientations of Figure~\ref{fig:2dcubes} with $V=\emptyset,i=1,j=2$. The eye and the bow have property~L. For the twin peak and the cycle, the L-graphs of $\emptyset$ each contain a cycle $1\rightarrow 2\rightarrow 1$, so the twin peak and the cycle do not satisfy property~L. Now consider Figure~\ref{fig:3dcubes}. The 3-cube orientation on the right satisfies property~L, but the left one, the spinner, does not because there is
a cycle $1\rightarrow 3\rightarrow 2\rightarrow 1$ in the L-graph of $\emptyset$.

Property~L turns out to be invariant under the reversal of the direction of all arcs along
a given set $R$ of dimensions. Let us formally introduce this operation,
as we will need it later.

\begin{definition}\label{def:reversal}
For a cube orientation $\O$ and $R\subseteq\carr\O$, let $\O\oplus
R$ be the cube orientation $\O'$ with $\vert\O'=\vert\O$ and outmap
\begin{equation}\label{eq:reverseR}
\phi_{\O'} (V) = \phi_{\O}(V)\oplus R, \quad V\in \vert\O'.
\end{equation}
\end{definition}

In Figure~\ref{fig:2dcubes}, we can see this in action: the twin peak
and the cycle can be obtained from each other by reversing
the edges along the vertical dimension. In contrast, the isomorphism
between the two $3$-cube orientations in Figure~\ref{fig:3dcubes} is
not of this type: it does reverse all edges in dimension $1$, but on
top of that, it also flips the two sides of the cube along this dimension.

\begin{observation}\label{obs:reversal} 
Let $\O$ be a cube orientation, $R\subseteq\carr\O$, and $\O' = \O\oplus
R$. Then $\O$ has property~L if and only if $\O'$ has property~L.
\end{observation}

\begin{proof}
For all $V\in \vert\O$ and $i\in\carr\O$, we have $\phi_{\O}(V)\oplus
\phi_{\O}(V\cup\{i\}) = \phi_{\O'}(V)\oplus
\phi_{\O'}(V\cup\{i\})$ as a consequence of (\ref{eq:reverseR}), so according to 
(\ref{eq:Lgraph}), $\O$ and $\O'$ have the same L-graphs.
\end{proof}

\section{Unique sink orientations}\label{sec:uso}
We turn our attention to \emph{unique sink orientations} (USOs), the subclass of cube orientations of interest. To define USOs, we need the concept of a face. 
  
\begin{definition}
  A \emph{face} (or subcube) of a cube orientation $\O$ is a directed subgraph of $\O$
  induced by an interval $\I\subseteq\vert\O$. A face is proper if
  $\I\neq\vert\O$.
\end{definition}

Faces are cube orientations themselves.

\begin{definition}
A cube orientation $\O$ is a \emph{unique
    sink orientation} (USO) if every face of $\O$ has a unique sink (a vertex with no outgoing edge).
\end{definition}

In particular, since the whole cube is itself a face, a USO has a
unique global sink. Among the four cube orientations depicted in
Figure~\ref{fig:2dcubes}, the eye and the bow are USOs, while the two
others are not USOs. The twin peak has two global sinks, while the cycle has
no global sink. (Faces of dimension $0$ (vertices) and $1$ (edges)
automatically have unique sinks.) 

The spinner in Figure~\ref{fig:3dcubes} (left) is an example of a
cyclic USO without property~L. But reversing all edges along one of
the dimensions yields an acyclic USO without property~L, by
Observation~\ref{obs:reversal}. On the other hand, the isomorphic copy
of the spinner in Figure~\ref{fig:3dcubes} (right) is cyclic but has
property~L. Hence, property~L (acyclicity of all L-graphs) is at least
not obviously related to acyclicity of the USO itself.

The first result of this paper is as follows. 

\begin{theorem}\label{thm:sufficient}
If a cube orientation $\O$ has property~L, then $\O$ is a USO.
\end{theorem}

We have already ``proved'' Theorem \ref{thm:sufficient} for the $2$-dimensional case with
Figure~\ref{fig:2dcubes}. Indeed, the 2-dimensional cube orientations with property
$L$ (the eye and the bow) are USOs. The converse of
Theorem~\ref{thm:sufficient} is false in dimension $n\geq 3$. For
example, the spinner in Figure~\ref{fig:uso} is a 3-dimensional USO
that does not have property~$L$.

For the proof of Theorem~\ref{thm:sufficient}, we employ the concept
of \emph{pseudo unique sink orientations} (pseudo USOs), which are minimal
non-USOs. A pseudo USO is a cube orientation that does not have a
unique global sink, but every proper face has a unique
sink~\cite[Definition 2]{PUSO}. Every cube orientation that is not
a USO contains a face that is a $m$-dimensional pseudo USO, for some $m\geq 2$.

It turns out that pseudo USOs have a surprisingly rigid structure. Among other properties, a pseudo USO has either no or exactly two global sinks~\cite[Corollary 6]{PUSO}; moreover, all outdegrees have the same parity~\cite[Lemma 8]{PUSO}. This allows us to prove the following

\begin{lemma}\label{lem:pseudo USOcycle}
Let $\O$ be a $m$-dimensional pseudo USO, $m\geq 3$, with $\vert\O=[V,W]$ and a global sink at $V$. Then $\O$ contains a directed cycle among the set of vertices of the form $V\cup\{i\}$ and $V\cup\{i,j\}$ for $i,j\in W\setminus V$. 
\end{lemma}

\begin{proof}
  Consider any vertex of the form $V\cup\{i,j\}$. Since the proper
  face of $\O$ induced by the interval
  $[V,V\cup \{i,j\}]=\{V, V\cup \{i\}, V\cup \{j\}, V\cup \{i,j\}\}$ has a
  unique sink (namely $V$), $V\cup\{i,j\}$ has an outgoing edge to one of
  $V\cup \{i\}$ and $V\cup \{j\}$. W.l.o.g.\
  suppose that it has an outgoing edge to $V\cup \{i\}$. The vertex $V\cup \{i\}$ in turn has an
  outgoing edge to the global sink $V$, but since all
  outdegrees in $\O$ are even ($V$ has outdegree $0$, and all outdegrees have the same parity), there is another outgoing edge to some $V\cup \{i,k\}$ for $k\neq j$ (the edge from $V\cup \{i,j\}$ was incoming). Now we repeat the argument from $V\cup \{i,k\}$. Continuing in this way, we eventually construct a directed cycle.
\end{proof}

\begin{proof}[Proof of Theorem~\ref{thm:sufficient}]
  We show the contraposition: if $\O$ is not a USO, then it has a cyclic L-graph, so $\O$ fails to have property~L. To this end, suppose that $\O$ is not a USO. Then $\O$ contains a face $\F$ that is a $m$-dimensional pseudo USO for some $m\geq 2$. Suppose that $\F$ is induced by the interval $[V,W]$. Since reversing edges along any set of dimensions $R$ neither affects property~L (Observation~\ref{obs:reversal}) nor the pseudo USO property~\cite[Lemma 8]{PUSO}, we may w.l.o.g.\ assume that $V$ is a global sink of $\F$. If $m=2$, then $\F$ is a twin peak, and the L-graph $\L_{\O}(V)$ contains a directed cycle of length $2$; see Figure~\ref{fig:2dcubes}. If $m\geq 3$, Lemma~\ref{lem:pseudo USOcycle} yields the existence of a directed cycle
  \[V\cup\{i_0\}\rightarrow V\cup\{i_0,i_1\} \rightarrow V\cup\{i_1\} \rightarrow \cdots  \rightarrow  V\cup\{i_{\ell}\} \rightarrow  V\cup\{i_{\ell},i_0\} \rightarrow V\cup\{i_0\}.
    \]
in $\F$. As $V$ is the sink of $\F$, the situation looks like Figure~\ref{fig:lcycle} for all $t=0,\ldots\ell$ (we define $i_{\ell+1}=i_0$). Consequently, $\L_{\O}(V)$ contains all the arcs $(i_t,i_{t+1}), t=0,\ldots,\ell$ and hence a directed cycle.
\begin{figure}[htb]
\begin{center}  
  \includegraphics[width=0.2\textwidth]{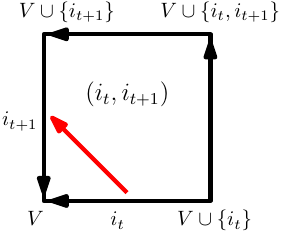}
\end{center}
\caption{Proof that $\L_{\O}(V)$ contains a directed cycle. \label{fig:lcycle}} 
\end{figure}
\end{proof}

We conclude this section by showing that there are doubly exponentially many $n$-cube USOs with property~L; hence,  in ``weeding out'' non-geometric USOs, our new property~L is not more efficient than the previously known ones.

\begin{theorem}\label{thm:counting}
There are at least $2^{2^n-1}$ $n$-cube USOs with property~L.
\end{theorem}

\begin{proof}
  The bound is attained by the class of \emph{recursively combed} $n$-cube USOs,
  which also provides the best known asymptotic lower bound for the
  number of acyclic USOs~\cite{Mat:The-Number}. A recursively combed
  $n$-cube USO $\O$ has all edges along dimension $n$
  oriented in the same way: either all of them go
  ``up'' ($V\rightarrow V\cup\{n\}$), or all of them go ``down''
  ($V\cup\{n\}\rightarrow V$). Moreover, the two facets with carrier
  $[n-1]$ (``lower'' and ``upper'' facet) are recursively combed USOs
  as well. For $n=1$, every USO is recursively combed. It follows that
  the number $r_n$ of recursively combed $n$-cube USOs satisfies the
  recurrence relation $r_1=2$ and $r_n=2r_{n-1}^2$. This solves to
  $r_n=2^{2^n-1}$.

  It is easy to see that all recursively combed USOs have
  property~L. Indeed, since all directed edges along dimension $n$ have the
  same direction, no L-graph contains any arcs of the form
  $i\rightarrow n$. As the arcs not involving $n$ are contained in
  L-graphs of the lower or the upper facet, it inductively follows
  that there are no arcs of the form $i\rightarrow j$ for
  $i<j$. Hence, all L-graphs are acyclic.
\end{proof}

The number of USOs is still significantly larger than
$2^{2^{n-1}}$, namely $n^{\Theta (2^n)}$~\cite{Mat:The-Number}. To
construct this many USOs, we can start with the \emph{uniform}
USO ($V\cup\{i\}\rightarrow V$ for all $V$ and all $i\notin V$) and
then reverse all edges in a matching. The result is called a
\emph{matching-reversal} USO. The lower bound then follows from a
bound for the number of matchings in the $n$-cube~\cite{Mat:The-Number}. We remark that this
construction may yield USOs without property~L; for example, the spinner in
Figure~\ref{fig:uso} is a
matching-reversal USO. We do not know whether the lower bound in
Theorem~\ref{thm:counting} can be significantly improved.

\section{P-cubes}\label{sec:P}
Stickney and Watson~\cite{stickney} were the first to show that every P-matrix
linear complementarity problem reduces to finding the global sink in a
USO. In this section, we briefly review their construction.

Given a matrix $M\in\R^{n\times n}$ and a vector $\qq\in\R^n$, the
  linear complementarity problem $\LCP(M,\qq)$ is to find vectors
  $\ww,\zz\in\R^n$ such that
  \begin{eqnarray*}
    \ww-M\zz&=&\qq, \\
    \ww,\zz&\geq& \nula, \\
    \ww^\top\zz &=& 0.
  \end{eqnarray*}
 While the first two conditions can be satisfied by solving a linear
 program, the third one makes the problem hard in general. More precisely, it is
 NP-complete to decide whether there is a
 solution~\cite{Chung:Hardness}.

 But if $M$ is a P-matrix (all principal minors are positive), the decision problem becomes
 trivial, because then there is a unique solution for every
 $\qq$~\cite{STW}. The problem of finding the unique vectors $\ww,\zz$
 has unknown complexity status. It is unlikely to be
 NP-hard~\cite{Meg:A-Note-on-the-Complexity}, but also no
 polynomial-time algorithm is known. The problem falls into a number
 of more recent complexity classes, namely PPAD, PLS, CLS, and
 UniqueEOPL~\cite{DBLP:journals/corr/abs-1811-03841}, but it is not
 known to be complete for any of these classes.

 The reduction to USO is as follows. First, observe that
 $\ww,\zz\geq\nula$ and $\ww^\top\zz=0$ together imply that for
 every $i\in[n]$, one of $w_i$ and $z_i$ is zero. Suppose that for some
 $V\subseteq[n]$, we set $w_i=0$ if $i\in V$ and $z_i=0$
 if $i\notin V$. Then there are unique values $w_i,$ for $i\notin V$ and $z_i$ for $i\in
 V$ such that $\ww-M\zz=\qq$. Indeed, taking the prescribed zeros into
 account, we must have $-M_{V,V}\zz_{V}=\qq_{V}$, where $M_{V,V}$ is the principal
 submatrix of $M$ with rows and columns indexed by $V$, and $\zz_V$ is the subvector
 of $\zz$ with entries indexed by $V$. 

 By definition of a P-matrix, $M_{V,V}$ has positive determinant and
 hence is invertible, so the missing (non-prescribed) $\zz$-entries
 $z_i$ for $i\in V$ are uniquely determined. This in turn determines
 the missing $\ww$-entries via $\ww=M\zz+\qq$.

 Hence, the problem of solving $\LCP(M,\qq)$ reduces to ``guessing'' 
 a set $S\subseteq[n]$ (an $n$-cube vertex) such that the missing entries of 
 $\ww$ and $\zz$ are nonnegative. 

 If $\qq$ is \emph{generic} (w.r.t.\ $M$), meaning that no missing entry is $0$, 
 we can turn this guesswork into finding the global sink in a suitably
 defined $n$-cube USO, with edge orientations defined by the signs of
 the missing entries. If $\zz$ and $\ww$ (uniquely) solve
 $\ww-M\zz=\qq; w_i=0, i\in V; z_i=0, i\notin V$, then we define
 \begin{equation}\label{eq:Parrows}
  V \rightarrow V\oplus\{i\} \quad \Leftrightarrow \quad z_i<0 \mbox{~or~} w_i<0.
\end{equation}
 To do this algebraically, we first formulate one system of
 equations that directly gives us the missing entries of
 $\ww$ and $\zz$ for a given $V$. When we write $\ww-M\zz=\qq$ as
 $I\ww-M\zz=\qq$, where $I$ is the $(n\times n)$ identity matrix,
 then we see that the $n$ missing entries can be obtained by solving
 $I_{\overline{V}}\ww_{\overline{V}}-M_V\zz_V =\qq$, where
 $\overline{V}=[n]\setminus V$ and matrix subscripts select colums. With $x_i=z_i$ if
 $i\in V$ and $x_i=w_i$ if $i\not\in V$, this is equivalent to solving
 the system $M(V)\xx=\qq$, where the $i$-th column $M(V)_i$ of matrix $M(V)$ is given by
 \begin{equation}\label{eq:AB}
   M(V)_i = \left\{\begin{array}{rl}
                     -M_i, & i\in V, \\
                     I_i, & i\notin V.
                   \end{array}\right.
               \end{equation}
               
As we know that the missing entries are uniquely
determined, the matrix $M(V)$ is invertible for all $V$. Hence, the outmap corresponding to
the orientation in (\ref{eq:Parrows}) is 
\begin{equation}\label{eq:Pcube}
  \phi(V) = \{i\in[n]: (M(V)^{-1}\qq)_i < 0\}, \quad V\subseteq[n].
\end{equation}
Stickney and Watson have shown that $\phi$ is the outmap of an
$n$-cube USO; its unique sink is the (unique) right guess
for $S$~\cite{stickney}.

\begin{definition}\label{def:P-cube}
A $n$-cube USO $\O$ is called a \emph{P-cube} if its outmap
is of the form given by (\ref{eq:AB}) and (\ref{eq:Pcube}) for some P-matrix
$M\in\R^{n\times n}$ and generic $\qq\in\R^n$.
\end{definition}

For example, the spinner in Figure~\ref{fig:uso} is a P-cube,
generated by
\begin{equation}\label{eq:spinnerMq}
 M = \left(\begin{array}{rrr}
              1 & 2& 0 \\
              0 & 1 & 2 \\
              2 & 0 & 1
            \end{array}\right),
  \quad \qq = \left(\begin{array}{r} 1 \\ 1 \\ 1
                    \end{array}\right).
\end{equation}

\section{D-cubes}\label{sec:D}
We now consider the subclass of P-cubes that arise from
\emph{symmetric} $P$-matrices. It turns out that these matrices are
exactly the positive definite ones, the symmetric matrices such that
$\xx^\top M\xx> 0$ for all $\xx\in\R^n$. The fact that a symmetric
P-matrix is positive definite follows from \emph{Sylvester's
  criterion}, according to which a matrix is positive definite if and
only if all its \emph{leading} principal minors are positive. Since a
symmetric P-matrix has all principal minors positive, it is positive definite. On the other hand, if we have a symmetric positive
definite matrix, then in fact \emph{all} principal minors are
positive, so we have a symmetric P-matrix. Indeed, by permuting rows
and columns in the same way, every principal minor can be made a leading
principal minor of another symmetric positive definite matrix and is hence
positive.

\begin{definition}
A $n$-cube USO $\O$ is called a \emph{D-cube} if its outmap
is of the form given by (\ref{eq:AB}) and (\ref{eq:Pcube}) for some
symmetric positive definite matrix $M\in\R^{n\times n}$ and generic $\qq\in\R^n$.
\end{definition}

Here is the main result of the paper:

\begin{theorem}\label{thm:main}
Every D-cube $\O$ has property~L.
\end{theorem}

This rules out that the spinner is a D-cube, but some isomorphic copies of it could still be D-cubes. Indeed, Stickney and Watson~\cite{stickney} have already shown that this is the case. With
\[
  M = \left(\begin{array}{rrr}
              5 &-10 & 2 \\
              -10 & 41 & -6 \\
              2 & -6 & 1
            \end{array}\right),
  \quad \qq = \left(\begin{array}{r} 1 \\ -7 \\ 1
                    \end{array}\right),
\]
we obtain the D-cube in Figure~\ref{fig:cyclic_D_cube}, an isomorphic copy of the spinner. 

\begin{figure}[htb]
\begin{center}  
  \includegraphics[width=0.2\textwidth]{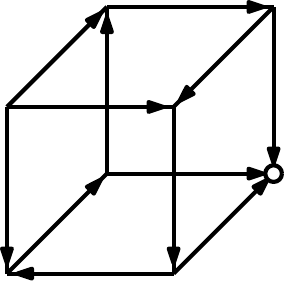}
\end{center}
\caption{A D-cube that is an isomorphic copy of the spinner (which is not a D-cube). \label{fig:cyclic_D_cube}} 
\end{figure}

Our main step towards the proof is Lemma \ref{lem:main} below. It generalizes an ad-hoc argument that was used to show that the spinner is not a D-cube~\cite{Sonoko}.
\begin{lemma}\label{lem:main}
If $\O$ is a D-cube, then the L-graph $\L_{\O}(\emptyset)$ is acyclic. 
\end{lemma}

\begin{proof}
Suppose that $\O$ is induced by $M=(m_{ij})$ and $\qq$, where $M$ is symmetric positive definite. 

Consider $i,j\in[n]$ such that $i\neq j$. In $\O$, we have
  \begin{equation}\label{eq:arcs}
    \begin{array}{rclccclcllc}
      \emptyset&\rightarrow&\{i\} &\quad & \Leftrightarrow &\quad & (M(\emptyset)^{-1}\qq)_i &=& q_i&<&0, \\
      \{j\}&\rightarrow&\{i,j\} &\quad & \Leftrightarrow &\quad & (M(\{j\})^{-1}\qq)_i &=& q_i - \frac{m_{ij}}{m_{jj}}q_j&<&0,
    \end{array}
  \end{equation} 
  as a consequence of 
\begin{equation}\label{eq:Mj}
M(\{j\})^{-1} = \left(\begin{array}{ccccc}
                1 &             & -m_{1j} & &  \\
                  &\ddots    &  \vdots  & &     \\
                  &             &  -m_{jj} & &  \\
                  &             &  \vdots  & \ddots & \\
                  &             &  -m_{nj} & & 1
              \end{array}
    \right)^{-1} = \left(\begin{array}{ccccc}
                1 &             & -m_{1j}/m_{jj} & &  \\
                  &\ddots    &  \vdots  & &     \\
                  &             &  -1/m_{jj} & &  \\
                  &             &  \vdots  & \ddots & \\
                  &             &  -m_{nj}/m_{jj} & & 1
              \end{array}
    \right).
\end{equation}
  
Suppose that $(s,t)$ is an arc in $\L_{\O}(\emptyset)$, meaning that -- by definition~(\ref{eq:Lgraph}) of the local graph -- exactly one of $\emptyset\rightarrow\{t\}$ and $\{s\}\rightarrow\{s,t\}$ holds. Applying (\ref{eq:arcs}) with $i=t$ and $j=s$, we see that this is equivalent to
\begin{equation}\label{eq:st}
q_t (q_t - \frac{m_{ts}}{m_{ss}}q_s) = q_t^2 - \frac{m_{ts}}{m_{ss}}q_tq_s< 0.
\end{equation}
(Here we also use that $\qq$ is generic, meaning that no expressions in (\ref{eq:arcs}) can be $0$.) As $\O$ is a USO, $(t,s)$ is not an arc in $\L_{\O}(\emptyset)$ (recall Figure~\ref{fig:2dcubes}). In the same vein as before, we get that this is equivalent to
\begin{equation}\label{eq:ts}
q_s (q_s - \frac{m_{st}}{m_{tt}}q_t) =  q_s^2 - \frac{m_{st}}{m_{tt}}q_sq_t> 0.
\end{equation}
  
Since $M$ is symmetric positive definite, we have that $m_{ss},m_{tt}>0$ and $m_{st}=m_{ts}$, so (\ref{eq:st}) and (\ref{eq:ts}) imply that 
\begin{equation} \label{eq:6}
0< m_{ss}q_t^2<m_{tt}q_s^2.
\end{equation} 
Now consider a path $s\rightarrow t\rightarrow u$ in $\L_{\O}(\emptyset)$. On top of (\ref{eq:6}) we then also get
\[
  0<m_{tt}q_u^2<m_{uu}q_t^2,
\]
and multiplying these inequalities gives 
\[
m_{ss}q_t^2 m_{tt}q_u^2 < m_{tt}q_s^2 m_{uu}q_t^2 \quad \Leftrightarrow \quad 
m_{ss}q_u^2 < q_s^2 m_{uu}.
\]
Iterating this, we get that (\ref{eq:6}) not only holds when $s\rightarrow t$ but actually
whenever there is a directed path from $s$ to $t$. This implies that there cannot be a directed path from a vertex $s$ back to $s$, meaning that $\L_{\O}(\emptyset)$ is acyclic.
\end{proof}

Now we are ready to prove the main theorem: every $D$-cube $\O$ has property~L.

\begin{proof}[Proof of Theorem~\ref{thm:main}]
  We proceed by induction on $n$. For $n\leq 2$, every USO has property~L, so we are done. For $n\geq 3$, suppose that every $(n-1)$-dimensional D-cube satisfies property~L. Let $\O$ be an $n$-dimensional D-cube  induced by $M$ and $\qq$, where $M$ is symmetric positive definite.

  The graph $\L_{\O}(\emptyset)$ is acyclic by Lemma~\ref{lem:main}, so it remains to verify that $\L_{\O}(V)$ is acyclic for given $V \neq \emptyset$. Choose $k \in V$. Notice that $V$ is in the $(n-1)$-dimensional face $\F$ of $\O$ induced by $[\{k\},[n]]$, and that $\L_{\O}(V)=\L_{\F}(V)$. Renaming dimensions $K=[n]\setminus\{k\}$ to $[n-1]$ turns $\F$ into an $(n-1)$-cube orientation $\O'$, and $V$ into $V'\subseteq[n-1]$, so that
$\L_{\O}(V)$ is isomorphic to $\L_{\O'}(V')$.

Orientation $\O'$ is known to be a P-cube, induced by $M'_{K,K}$ (and suitable $\qq'_K$), where
  \begin{equation}\label{eq:Mprime}
    M'=M(\{k\})^{-1}M,
  \end{equation}
  with $M(\{k\})^{-1}$ as in (\ref{eq:Mj})~\cite[Property~5]{stickney}. We also claim that $M'_{K,K}$ is symmetric positive definite, which yields the desired result because it implies that $\O'$ is actually a D-cube of dimension $n-1$ that has property~L by the inductive hypothesis. Hence $\L_{\O'}(V')$ and $\L_{\O}(V)$ are acyclic, as desired. 

To show that $M'_{K,K}$ is symmetric positive definite, we first observe (simple calculations) that (\ref{eq:Mprime}) yields $M'=(m'_{ij})$ with 
\begin{equation}
m'_{ij} = m_{ij}-\frac{m_{ik}m_{kj}}{m_{kk}}, \quad i,j\neq k.
\end{equation}
Since $M$ is symmetric, it follows that $M'_{K,K}$ is symmetric. Furthermore, for $\xx\in\R^{n-1}$, we can easily verify that $\xx^T M'_{K,K}\xx = \yy^TM\yy\geq 0$, where $\yy=(x_1,\ldots,x_{k-1}, -(\sum_{i\neq k}m_{ki}x_i)/m_{kk}, x_{k+1},\ldots,x_n)$, and thus $M'_{K,K}$ is also positive definite.
\end{proof}
 
\section{Kaleidoscopes}\label{sec:Dmisc}

We have seen that property~L is in general not preserved under the application of a cube automorphism; see Figure~\ref{fig:3dcubes}. This begs the question: does every USO have an isomorphic copy that satisfies property~L? If so, property~L would in the above sense be satisfied by all USOs, and it would not really be justified to call the property ``combinatorial''.

In dimensions 3 and 4, the question can be answered by brute-force, using lists of all isomorphism classes of USOs. In dimension 3, there are $19$ such classes, and each one turns out to contain a member with property~L. But already in dimension 4, there are $9$ (out of $14614$) isomorphism classes that contain no member with property~L. Figure~\ref{fig:4d_non_L} depicts a 4-dimensional USO such that no isomorphic copy has property~L. 

\begin{figure}[htb]
\begin{center}  
  \includegraphics[width=0.5\textwidth]{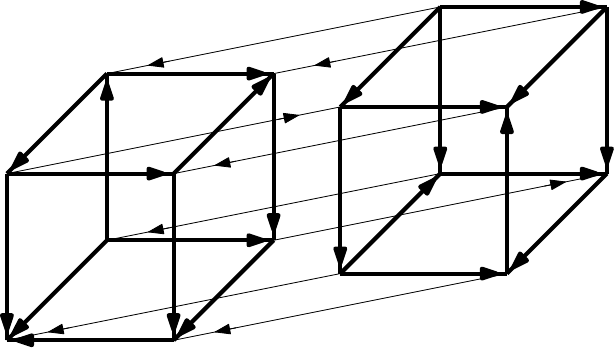}
\end{center}
\caption{A USO with no isomorphic copy having property~L.\label{fig:4d_non_L}} 
\end{figure}

In this section, we systematically construct such examples in higher dimensions. 
We show that for every $n$-cube USO, there is a $2n$-cube USO---a \emph{kaleidoscope}---that contains all ``mirror images'' of the former. And using this fact, we show that if a $n$-cube USO does not have property~L, then no isomorphic copy of the corresponding kaleidoscope has property~L. Separately, starting from any P-cube, we show how to construct a kaleidoscope that is also a P-cube. In particular, starting from the spinner, we construct a 6-dimensional P-cube with no isomorphic copy that satisfies property~L. 

Throughout this section, we mostly consider $n$-cube USOs and $2n$-cube USOs. Whenever we consider USOs of (some possibly other) general dimension, we will use $m$ to denote this dimension. It will also be convenient to slightly abuse notation and identify a USO with its outmap $\phi$. 

Two $m$-cube USOs $\psi$ and $\psi'$ are \emph{isomorphic} if there is a bijection $h:2^{[m]}\rightarrow 2^{[m]}$ (a cube automorphism) such that for all $V,V'\subseteq[m]$, we have $V\rightarrow_{\psi} V'$ if and only if $h(V)\rightarrow_{\psi'}h(V')$. For example, the map defined by $h(V)= V\oplus F$, for some $F \subseteq [m]$, is a cube automorphism. It maps a USO to one of its $2^m$ \emph{mirror images}: 

\begin{definition}\label{def:mirror}
  Let $\psi$ be an $m$-cube USO and $F\subseteq[m]$. The USO $\psi'$ defined by
  \begin{equation}\label{eq:mirror}
    \psi'(V) = \psi(V\oplus F), \quad V\subseteq [m],
  \end{equation}
is the \emph{mirror image} of $\psi$ along dimensions $F$.
\end{definition}

Figure \ref{fig:mirrors} depicts the $8$ mirror images of the spinner (one of which is the spinner itself, corresponding to $F = \emptyset$). Note that mirroring is not to be confused with reversing edges, $\psi'(V) = \psi(V)\oplus R$ (Definition~\ref{def:reversal}). While the latter operation also preserves the USO property~\cite[Lemma 2.1]{szabo2001unique}, it is in general not an automorphism.

\begin{figure}[htb]
\begin{center}  
  \includegraphics[width=0.7\textwidth]{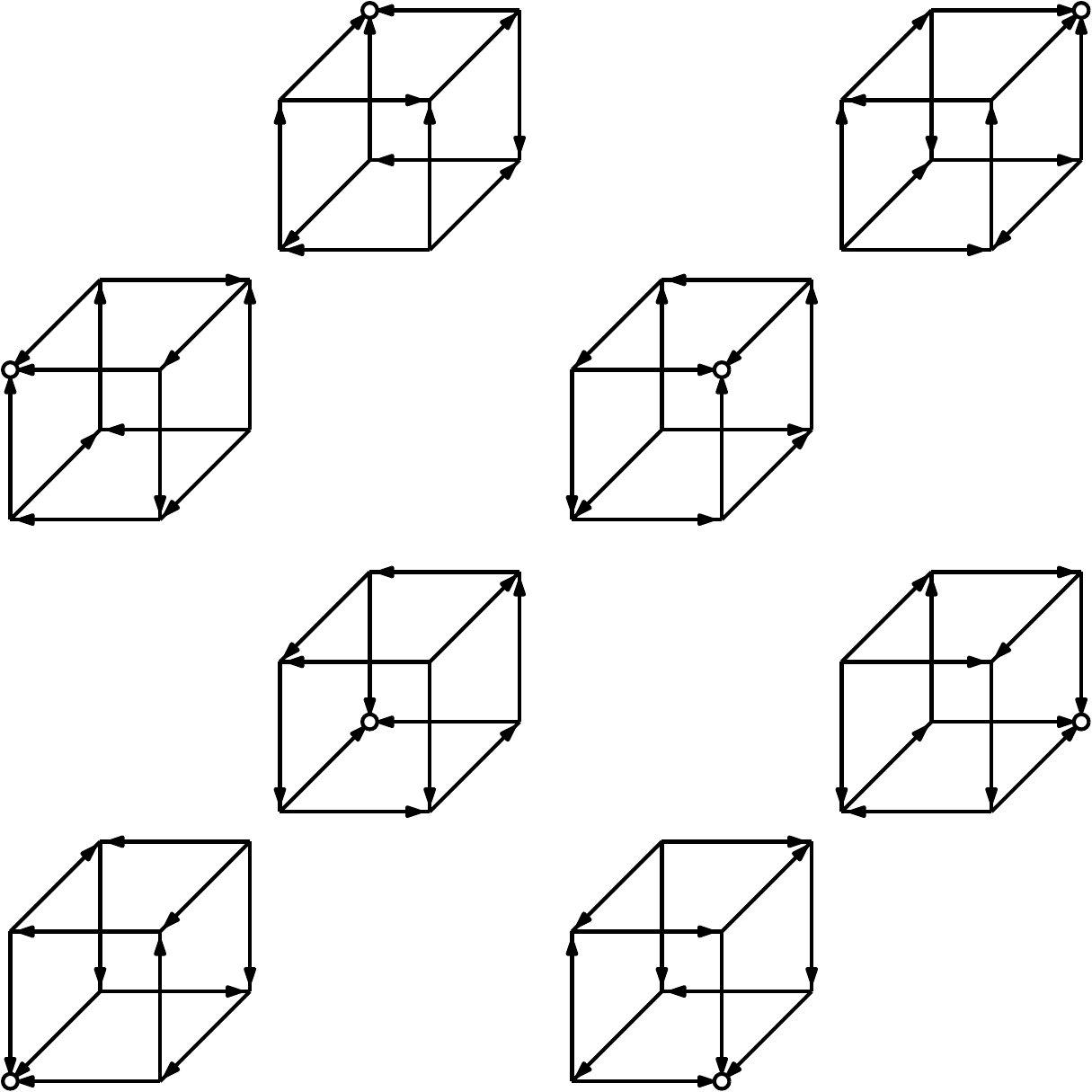}
\end{center}
\caption{The $8$ mirror images of the spinner (lower left corner).\label{fig:mirrors}} 
\end{figure}

Actually, all automorphisms can be described as a mirroring automorphism up to a permutation of the dimensions. That is, any $m$-cube automorphism $h$ is of the form
$$h(V)=\pi (V \oplus F), \quad V\subseteq[m],$$
where $F\subseteq[m]$ is the set of mirrored dimensions, and
$\pi:[m]\rightarrow[m]$ is a permutation (renaming) of the dimensions. This follows from the fact that a cube automorphism is determined by the images of the $m+1$ vertices $\emptyset$ and $\{i\}$ for $i\in[m]$.

Next we define the concept of a kaleidoscope, a USO that connects all mirror images of a given $n$-cube USO along $n$ new dimensions. For this, we introduce the following notation.
For a vertex (or outmap value) $V\subseteq[2n]$, $V_L=V\cap[n]$ denotes the lower dimensions, and $V_H=\{i-n: i\in V\cap\{n+1,\ldots,2n\}\}$ the upper dimensions, renamed such they also fall into the range $[n]$. For example, if $n=3$ and $V=\{1,2,4,6\}$, then $V_L=\{1,2\}$ and $V_H=\{1,3\}$. Note that $(U\oplus V)_L=U_L\oplus V_L$ and  $(U\oplus V)_H=U_H\oplus V_H$; we will use this in arguments below.

\begin{definition}
  Let $\phi$ be an $n$-cube USO and $\psi$ a $2n$-cube USO. The USO $\psi$ is a \emph{kaleidoscope} for $\phi$ if
  \begin{equation}\label{eq:kaleidoscope}
    \psi(V)_L = \phi (V_L \oplus V_H), \quad \forall~ V\subseteq [2n].
  \end{equation}
\end{definition}

Figure~\ref{fig:kaleidoscope} illustrates a kaleidoscope for the spinner, connecting the mirror images in Figure~\ref{fig:mirrors} along new dimensions $4,5,6$. In general, each of the $2^n$ subcubes with carrier $[n]$ of a kaleidoscope is a particular mirror image of $\phi$, so that $\psi$ contains all possible mirror images of $\phi$. More precisely, if $V\subseteq\{n+1,\ldots,2n\}$, then the subcube induced by the interval $[V,V\cup[n]]$ is the mirror image of $\phi$ along dimensions $V_H$. We will make this containment formal in Definition~\ref{def:contains} and Lemma~\ref{lem:kaleidoscope} below.

\begin{figure}[htb]
\begin{center}  
  \includegraphics[width=0.7\textwidth]{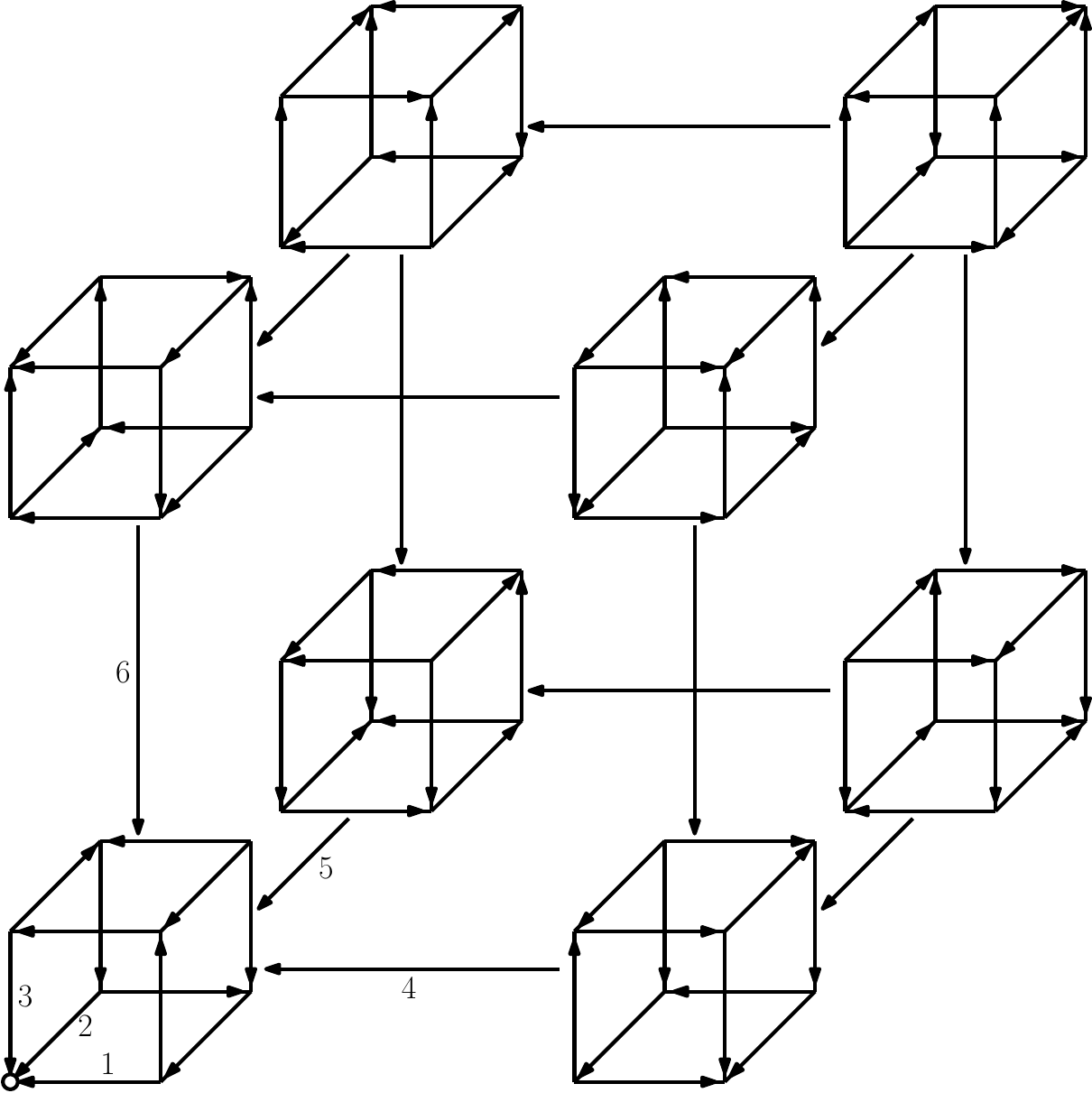}
\end{center}
\caption{The kaleidoscope for the spinner resulting from Lemma~\ref{lem:product_kaleidoscope}; all edges along dimensions $4,5,6$ are directed ``downwards'' (from larger to smaller sets).\label{fig:kaleidoscope}} 
\end{figure}

We remark that for every $\phi$, we can construct a kaleidoscope $\psi$. For example, if we set $\psi(V)_H=V_H$ for all $V$, this defines $\psi$ completely, together with (\ref{eq:kaleidoscope}), and the result is a USO; this is a special case of the \emph{product construction} due to Schurr and Szab\'o~\cite[Lemma 3]{SchSza:Finding}. For completeness, we provide a proof.

\begin{lemma}\label{lem:product_kaleidoscope}
  Let $\phi$ be a $n$-cube USO and define $\psi:2^{[2n]}\rightarrow 2^{[2n]}$
  by
  \[
  \begin{array}{rcl}
    \psi(V)_L &:=& \phi (V_L \oplus V_H),\\
    \psi(V)_H &:=& V_H. 
  \end{array}, \quad V \subseteq[2n].
\]
Then $\psi$ is a $2n$-cube USO and hence a kaleidoscope for $\phi$.
\end{lemma}

\begin{proof}
  According to \cite[Lemma 2.3]{szabo2001unique}, $\psi$ is a USO if and only if $(\psi(U)\oplus\psi(V))\cap (U\oplus V)\neq \emptyset$ for all $U\neq V$. In words, restricted to the subcube spanned by $U$ and $V$, the two vertices have different outmap values. We verify this property for $\psi$ as defined above. Fix $U\neq V$. If $U_H\neq V_H$, choose $i\in U_H\oplus V_H=\psi(U)_H\oplus \psi(V)_H=(\psi(U)\oplus\psi(V))_H=(U\oplus V)_H$. Hence, $i+n\in (\psi(U)\oplus \psi(V))\cap (U\oplus V))$. If $U_H=V_H=:W$, we have $U_L\neq V_L$ and $U_L\oplus W =: U' \neq V' := V_L\oplus W$. Since $\phi$ is a USO, there is $i\in (\phi(U')\oplus \phi(V'))\cap (U'\cap V')$, equivalently, 
  \[i\in (\psi(U)_L \oplus \psi(V)_L) \cap  (U_L\oplus V_L) = (\psi(U) \oplus \psi(V))_L \cap
 (U\oplus V)_L\subseteq (\psi(U)\oplus\psi(V))\cap (U\oplus V).\]
In both cases, $(\psi(U)\oplus\psi(V))\cap (U\oplus V)\neq \emptyset$.
\end{proof}

\begin{definition}\label{def:contains}
  Let $\phi$ be an $n$-cube USO and $\psi'$ a $2n$-cube USO. The USO $\psi'$ \emph{contains} $\phi$ if there is $V\subseteq\{n+1,\ldots,2n\}$ such that 
  \begin{equation}\label{eq:contains}
  \psi'(U)_L = \phi(U_L), \quad \forall~U\in [V,V\cup[n]].
  \end{equation}
\end{definition}
This means that the subcube of $\psi'$ induced by $[V,V\cup[n]]$ is a  ``translated copy'' of $\phi$.

\begin{lemma}\label{lem:kaleidoscope}
Suppose that $\psi$ is a kaleidoscope for the $n$-cube USO $\phi$. Let $F\subseteq[2n]$, and let $\psi'$ be the mirror image of $\psi$ along dimensions $F$. Then $\psi'$ contains $\phi$.
\end{lemma}
\begin{proof}
  Let $V\subseteq \{n+1,\ldots,2n\}$ be the unique set such that $V_H=F_L\oplus F_H$. We claim that the subcube induced by $[V,V\cup[n]]$ is the one carrying the translated copy of $\phi$. To check this, fix $U\in [V,V\cup[n]]$ and set $W=U\oplus F$. Then
$W_L=U_L\oplus F_L$ and $W_H=U_H\oplus F_H$. Furthermore,
  \begin{eqnarray*}
    \psi'(U)_L &\stackrel{(\ref{eq:mirror})}{=}& \psi (U\oplus F)_L = \psi (W)_L  \\
               &\stackrel{(\ref{eq:kaleidoscope})}{=}& \phi (W_L\oplus W_H)
                = \phi (U_L \oplus U_H \oplus F_L \oplus F_H) \\
                &=& \phi (U_L \oplus U_H \oplus V_H) = \phi (U_L),
  \end{eqnarray*}
using $U_H=V_H$ for $U\in [V,V\cup[n]]$. Hence, we have verified (\ref{eq:contains}). 
\end{proof}

Here is the first main result of this section.

\begin{theorem}\label{thm:kaleidoscope}
Suppose the $2n$-cube USO $\psi$ is a kaleidoscope for the $n$-cube USO $\phi$, and further suppose that $\phi$ fails to have property~L. Let $\psi'$ be any USO isomorphic to $\psi$. Then $\psi'$ does not satisfy property~L either.
\end{theorem}

\begin{proof}
Recall that any $m$-cube automorphism $h$ is of the form
\begin{equation}\label{eq:h}
 h(V)=\pi (V \oplus F), \quad V\subseteq[m],
\end{equation}
where
$F\subseteq[m]$ is the set of mirrored dimensions, and
$\pi:[m]\rightarrow[m]$ is a permutation (renaming) of dimensions.

With $m=2n$, let $\psi$ and $\psi'$ be isomorphic under $h$ as in (\ref{eq:h}). We first consider the case where $\pi=id$, the identity permutation. In this case, $\psi'$ is the mirror image of $\psi$ along dimensions $F$, and by Lemma~\ref{lem:kaleidoscope}, $\psi'$ contains $\phi$ in the sense that there is some $V\subseteq\{n+1,\ldots,2n\}$ such that (\ref{eq:contains}) holds. This shows that the L-graph $\L_{\psi'}(V\cup W)$ contains $\L_{\phi}(W)$ as a subgraph. Indeed, $(i,j)$ is by (\ref{eq:Lgraph}) an arc of $\L_{\phi}(W)$ if and only if $j\in \phi (W)\oplus \phi (W\cup \{i\}) = \phi ((V\cup W)_L)\oplus \phi ((V\cup W\cup \{i\})_L) $. By (\ref{eq:contains}) applied with $U=V\cup W, V\cup W\cup \{i\}$, we then have
$j\in \psi'(V\cup W)_L \oplus \psi'(V\cup W\cup\{i\})_L$ which for $j\in[n]$ is equivalent to $j\in \psi'(V\cup W) \oplus \psi'(V\cup W\cup\{i\})$, so $(i,j)$ is also an arc of $\L_{\psi'}(V\cup W)$. Since $\L_{\phi}(W)$ is cyclic, so is $\L_{\psi'}(V\cup W)$, and $\psi'$ does not have property~L.

In the general case where $\pi\neq id$, we obtain $\psi'$ by first mirroring $\psi$ along the dimensions $F$ (the result does not have property~L), and then renaming dimensions according to $\pi$. As the latter operation does not affect property~L, $\psi'$ does not have property~L, either.
\end{proof}

\paragraph{P-cube kaleidoscopes.} From Theorem~\ref{thm:kaleidoscope}, we already know that there are USOs without an isomorphic copy that satisfies property~L, namely any kaleidoscope of a $n$-cube USO without property~L. Here we consider the setting in which the $n$-cube USO is a P-cube. Our goal is to start with the P-cube and construct a kaleidoscope that is also a P-cube, implying the existence of P-cubes that have no isomorphic copy with property~L.

It is not a priori clear how to approach this contruction because it is not known if the kaleidoscopes built from Schurr and Szab\'o's combinatorial product construction~\cite{SchSza:Finding} are P-cubes. In what follows, we develop an algebraic construction that builds a P-cube kaleidoscope from any given P-cube. We found it somewhat surprising that this works (in a simple way), and the construction may be of general interest as a new way to build P-cubes from P-cubes.

We start with a construction that ``blows up'' a P-matrix to twice its dimension.

\begin{lemma}\label{lm:general}
Let $A$ be an $n\times n$ P-matrix. Then \[
   M = \left(
     \begin{array}{cc}
       A & A + I \\
       A-I & A
     \end{array}
     \right) \in \R^{2n\times 2n}
   \]
   is also a P-matrix.
 \end{lemma}

 \begin{proof}
   We use a known characterization of P-matrices~\cite[Theorem
   3.3.4]{CotPanSto:LCP}: $M$ is a P-matrix if and only if for every
   nonzero vector $\zz$, there exists an index $i$ such that
   $z_i(M\zz)_i>0$ ($M$ does not reverse all signs of $\zz$). 

   So let us consider a nonzero $2n$-vector $\zz=(\xx,\yy)$ where
   $\xx$ and $\yy$ are $n$-vectors. For $i\in[n]$, we have
   \begin{equation}\label{eq:extend_P1}
     z_i(M\zz)_i = x_i (A\xx + (A+I)\yy)_i = x_i (A(\xx+\yy))_i + x_iy_i,
   \end{equation}
   and 
   \begin{equation}\label{eq:extend_P2}
     z_{i+n}(M\zz)_{i+n} = y_i ((A-I)\xx + A\yy)_i = y_i (A(\xx+\yy))_i - x_iy_i,
   \end{equation}
   If $\xx=-\yy\neq \nula$, there is an index $i$ such that $x_i=-y_i\neq
   0$. In this case,  (\ref{eq:extend_P2}) yields $
   z_{i+n}(M\zz)_{i+n}=-x_iy_i>0$. If $\xx\neq -\yy$, we 
   add up (\ref{eq:extend_P1}) and  (\ref{eq:extend_P2}) to get
   \[
z_i(M\zz)_i + z_{i+n}(M\zz)_{i+n} = (\xx+\yy)_i (A(\xx+\yy))_i.
\]
and since $A$ is a $P$-matrix, there is an index $i$ such that
$z_i(M\zz)_i + z_{i+n}(M\zz)_{i+n} >0$. But then, one of $z_i(M\zz)_i$
and $z_{i+n}(M\zz)_{i+n}$ must be positive as well. 
\end{proof}

Using the blow-up construction of Lemma \ref{lm:general}, along with a suitable right-hand side $\qq$, we can now construct a P-cube kaleidoscope from any P-cube. 

\begin{theorem}\label{thm:pkaleidoscope}
  Let $A\in\R^{n\times n}$ be a $P$-matrix, $\bb\in\R^n$ generic. Let $\phi$ be the $n$-dimensional P-cube defined by $A$ and $\bb$; see Definition~\ref{def:P-cube}. Let 
  \[
     M = \left(
     \begin{array}{cc}
       A & A + I \\
       A-I & A
     \end{array}
     \right) \in \R^{2n\times 2n}
   \]
   and
   \[
     \qq = \left(\begin{array}{cc} \bb \\ \bb 
                  \end{array}\right) \in \R^{2n}.
                \]
Then $M$ and $\qq$ define a $2n$-dimensional P-cube $\psi$. Furthermore, $\psi$ is a kaleidoscope for $\phi$.
\end{theorem}

\begin{proof}
  We already know from Lemma~\ref{lm:general} that $M$ is a P-matrix. If $\qq$ is generic, $M$ and $\qq$ define a P-cube $\psi$. For the kaleidoscope property, we will compute the outmap values $\psi(V), V\subseteq[2n]$. In doing so, we will also see that $\qq$ is indeed generic. Let us fix $V$ for the remainder of the proof. We recall from Section~\ref{sec:P} that $\psi(V)$ is determined by the signs of the variables not prescribed by $V$. More concretely, there are unique $\ww$ and $\zz$ that solve the system of equations $\ww-M\zz=\qq$; $w_i=0, i\in V$; $z_i=0, i\notin V$. And the outmap value of $V$ is
  \begin{equation}\label{eq:psigeneral}
\psi(V) = \{i: z_i<0 \mbox{~or~} w_i<0 \},
\end{equation}
where $\qq$ being generic means that $z_i=w_i=0$ can never happen. 
Hence, for $M$ and $\qq$ as in the statement of the theorem, there are unique $\ww,\ww',\zz,\zz'\in\R^n$ such that
\begin{eqnarray}
\left(\begin{array}{cc} \ww \\ \ww'
      \end{array}\right) -  \left(
     \begin{array}{cc}
       A & A + I \\
       A-I & A
     \end{array}
     \right) \left(\begin{array}{cc} \zz \\ \zz'
      \end{array}\right)  &=& \left(\begin{array}{cc} \bb \\ \bb 
                     \end{array}\right), \label{eq:2neq}\\
  w_i&=&0, \quad i\in V_L,  \label{eq:2neq1} \\ 
  w'_i&=&0, \quad i\in V_H, \label{eq:2neq2} \\
  z_i&=&0, \quad i\notin V_L, \label{eq:2neq3}  \\
  z'_i&=&0, \quad i\notin V_H. \label{eq:2neq4}
\end{eqnarray}
The outmap (\ref{eq:psigeneral}) is then determined by 
\begin{eqnarray}
  \psi(V)_L &=& \{i\in[n]: z_i<0 \mbox{~or~} w_i<0\}, \label{eq:psispecial1} \\ 
  \psi(V)_H &=& \{i\in[n]: z'_i<0 \mbox{~or~} w'_i<0\}. \label{eq:psispecial2}
\end{eqnarray}
Expanding (\ref{eq:2neq}), we get
\begin{eqnarray}
  \ww  -\zz' - A(\zz + \zz') = \bb, \\ \label{eq:lcpA1}
  \ww' + \zz - A(\zz + \zz') = \bb, \label{eq:lcpA2}
\end{eqnarray}
and it follows that
\begin{eqnarray}
  \ww-\zz'=\ww'+\zz &:=& \yy, \label{eq:ydef}\\
  \zz+\zz'=\ww-\ww' &:=& \xx. \label{eq:xdef}
\end{eqnarray}

With this and (\ref{eq:2neq1}) through (\ref{eq:2neq4}), we can summarize the situation as in Figure~\ref{fig:kproof}.

\begin{figure}[htb]
\begin{center}  
  \includegraphics[width=0.5\textwidth]{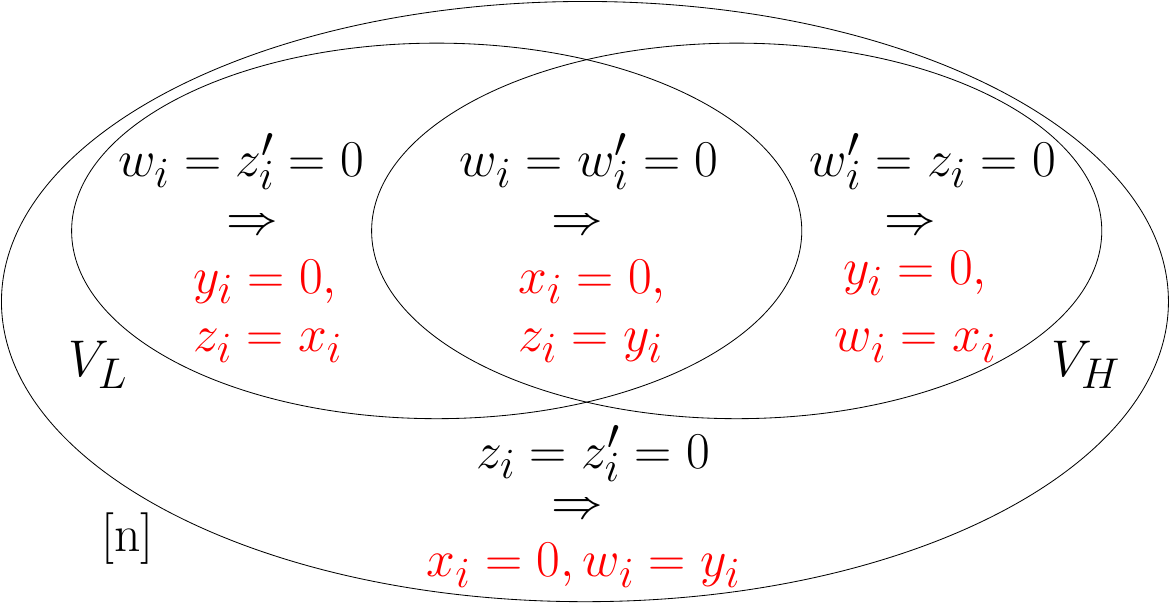}
\end{center}
\caption{Proof of Theorem~\ref{thm:pkaleidoscope}: Prescribed variables and implied equalities\label{fig:kproof}} 
\end{figure}

From (\ref{eq:ydef}), (\ref{eq:xdef}), and Figure~\ref{fig:kproof}, we see that $\xx$ and $\yy$ (uniquely) solve the system $\yy-A\xx=\bb; y_i=0, i\in V_L\oplus V_H; x_i=0, i\notin V_L\oplus V_H$. Hence, 
the P-cube $\phi$ determined by $A$ and $\bb$ by definition satisfies
\begin{equation}\label{eq:phiV}
\phi (V_L\oplus V_H) = \{i\in[n]: x_i < 0 \mbox{~or~} y_i <0\}.
\end{equation}

To conclude that $\psi$ is a kaleidoscope for $\phi$, it remains to show that $\qq$ is generic, and that the kaleidoscope property (\ref{eq:kaleidoscope}) holds, meaning the outmap value in (\ref{eq:phiV}) equals the one in (\ref{eq:psispecial1}). For the latter, we have to prove that
\begin{equation}\label{eq:xyzw}
x_i < 0 \mbox{~or~} y_i <0 \quad \Leftrightarrow\quad  z_i<0 \mbox{~or~} w_i<0.
\end{equation}
This easily follows from the equalities depicted in Figure~\ref{fig:kproof}. If $x_i<0$, we have $z_i=x_i<0$ or $w_i=x_i<0$. If $y_i<0$, we have $z_i=y_i<0$ or $w_i=y_i<0$. Vice versa, if $z_i<0$, then $x_i=z_i<0$ or $y_i=z_i<0$; and if $w_i<0$, then $x_i=w_i<0$ or $y_i=w_i<0$.

To prove that $\qq$ is generic, we need to show that $z_i=w_i=0$ and $z'_i=w'_i=0$ can never happen. Adding up (\ref{eq:ydef}) and (\ref{eq:xdef}) yields $\ww+\zz=\xx+\yy$. Since $\bb$ is generic, we know that $x_i+y_i\neq 0$ for all $i$ (one of the values is $0$, the other one isn't). Hence, also $w_i+z_i\neq 0$ for all $i$. Again, one of these two values is $0$, so $w_i-z_i\neq 0$ for all $i$ as well, and with (\ref{eq:ydef}), $z'_i+w'_i\neq 0$ follows for all $i$.
\end{proof}

As an example, the kaleidoscope for the spinner, depicted in Figure~\ref{fig:kaleidoscope}, is generated by
\[
M= \left(
\begin{array}{cccccc}
1 & 2 & 0 & 2 & 2 & 0 \\
0 & 1 & 2 & 0 & 2 & 2 \\
2 & 0 & 1 & 2 & 0 & 2 \\
0 & 2 & 0 & 1 & 2 & 0 \\
0 & 0 & 2 & 0 & 1 & 2 \\
2 & 0 & 0 & 2 & 0 & 1
\end{array}
\right), \quad \qq= \left(
\begin{array}{c} 1 \\ 1 \\ 1 \\ 1 \\ 1 \\ 1
\end{array}
\right),
\]
following the construction of Theorem~\ref{thm:pkaleidoscope} for the matrix $A$ and vector $\bb$ that generate the spinner according to (\ref{eq:spinnerMq}).

\section{Outlook}\label{sec:conclusion}
The main research question this work raises is whether we can algorithmically make use of property~L. Ideally, we would like to exploit the property towards speeding up sink-finding for D-cubes. 

On the one hand, it seems discouraging that property~L is in general not invariant under applying cube automorphisms. In the ``abstract USO world'', pairs of isomorphic USOs are typically considered ``the same'', and many algorithms in fact perform on them in the same way. Such algorithms do not have much potential to exploit the input's property~L when another isomorphic copy fails to satisfy it. 

On the other hand, the insight here may be that algorithms tailored towards D-cubes should have a ``sense of direction'' and \emph{not} ignore how vertices are labeled. 

There is a family of USO algorithms commonly summarized under the term \emph{product algorithm}~\cite[Lemma 3.2]{szabo2001unique}. In order to be able to apply them to a class $\mathcal{U}$ of USOs, two properties of $\mathcal{U}$ are required: $\mathcal{U}$ must be closed under taking subcubes, but also under taking \emph{inherited orientations}~\cite[Section 3]{szabo2001unique}. The class of all USOs
meets both requirements. The class of D-cubes is closed under taking subcubes (this implicitly follows from the proof of Theorem~\ref{thm:main}); but it is unknown whether the class is closed under taking inherited orientations (probably, it's not). 

Finding weaker closure properties and corresponding algorithms with a ``sense of direction''   might be a way to make progress here. 

\bibliography{lamperski.bib}
\bibliographystyle{plain}

\end{document}